\documentclass[a4paper,11pt]{article}
\usepackage[applemac]{inputenc}
\usepackage[T1]{fontenc}

\usepackage{latexsym,amsthm}
\usepackage{amssymb}
\usepackage{amsfonts, amsmath}
\usepackage{amscd}
\usepackage{amstext}
\input{xypic.tex} \xyoption{all}
\usepackage[pdftex]{color, graphicx}
\usepackage{makeidx}
\usepackage{geometry}
\usepackage{indentfirst}

\newtheorem{Theorem}{Theorem}[section]
\newtheorem{Lemma}[Theorem]{Lemma}
\newtheorem{Proposition}[Theorem]{Proposition}
\newtheorem{Corollary}[Theorem]{Corollary}
\newtheorem{Definition}[Theorem]{Definition}

\newtheorem*{Conjecture}{Conjecture}

\theoremstyle{definition}

\newtheorem{Remark}[Theorem]{Remark}

\newcommand{\opn}{{\mathcal{O}_{\mathbb{P}^n}}}
\newcommand{\pn}{{\mathbb{P}^n}}

\newcommand{\KK}{{\mathbb K}}

\newcommand{\PP}{{\mathbb P}}

\newcommand{\OO}{{\mathcal O}}

\newcommand{\rk}{\operatorname{rk}}
\newcommand{\Hom}{\operatorname{Hom}}

\newcommand{\Ext}{\operatorname{Ext}}

\makeindex

\begin{document}

\title{Simplicity and exceptionality of syzygy bundles over $\PP^n$}
\author{Simone Marchesi, Daniela Moura Prata}
\date{}
\maketitle
\abstract{\noindent In this work we will prove results that ensure the simplicity and the exceptionality of vector bundles which are defined by the splitting of pure resolutions. We will call such objects syzygy bundles. }

\section*{Introduction}

The study of particular families of vector bundles over projective varieties has always taken a great part in algebraic geometry. In particular, many authors focused on the family of syzygy bundles, defined as the kernel of an epimorphism of the form $$\phi: \oplus_{i=1}^t \OO_{\PP^n}(-d_i) \longrightarrow \OO_{\PP^n},$$ that have been studied in the last decades. Brenner in \cite{B} gives combinatorial conditions for (semi)stability of the syzygy bundles when they are given by monomial ideals. Co$\rm{\check{a}}$nda in \cite{Co} studies stability for syzygies defined by polynomials of the same degree, of any possible rank for $n \geq 3$. Costa, Marques, Mir\'o-Roig, see \cite{CMMR}, also study stability of syzygies given by polynomials of same degree and studied moduli spaces.

Ein, Lazarsfeld and Mustopa in \cite{ELM, EL} extend the problem for smooth projective varieties X, studying the stability of the syzygy bundles that are given by the kernel of the evaluation map $\mbox{eval}_L: H^0(L)\otimes_{\KK} \mathcal{O}_X \longrightarrow L$ where $L$ is a very ample line bundle over $X$. 

We define the \emph{syzygy bundles} as the vector bundles coming from the splitting of pure resolutions of the form

\begin{eqnarray}\label{pres}
\xymatrix{0 \ar[r] & \OO_{\PP^n}^{\beta_{p}}(-d_{p})  \ar[r] & \cdots \ar[r] & \OO_{\PP^n}^{\beta_{1}}(-d_{1}) \ar[r] & \OO_{\PP^n}^{\beta_{0}} \ar[r] & 0}
\end{eqnarray}  into short exact sequences. Observe that the first syzygy bundle $F$ in (\ref{pres}), obtained as 
$$
0\longrightarrow F \longrightarrow \OO_{\PP^n}^{\beta_{1}}(-d_{1}) \longrightarrow \OO_{\PP^n}^{\beta_{0}} \longrightarrow 0
$$
is also a syzygy bundle in the sense of \cite{B}, \cite{Co} and \cite{CMMR}.


In the first section, we recall some notions on pure resolutions and we introduce in detail what we will mean by \emph{syzygy bundle} on the projective space.

In the second section, we will prove two results, see Theorem \ref{thm-b0} and Theorem \ref{thm-f1}, which ensure the simplicity of the syzygy bundles previously defined. The results here generalize the ones proved in Section 4 and Section 5 of \cite{Prata}.

In particular we provide an answer to a question proposed by Herzog and K$\rm{\ddot u}$hl in \cite{herkuh}, where they wonder whether the modules coming from linear pure resolution of monomials ideals are indecomposable or not. We will be able to ensure such property under specific hypotheses, see Remark \ref{rmk}.

In the third section, we will show necessary and sufficient conditions to prove their exceptionality, see Theorem \ref{thm-excep}, and we will state a conjecture which relates syzygy bundles with Steiner bundles. 

In the fourth section, we will consider some classical pure resolutions, studying when the bundles defined in their splitting are simple and when exceptional.
\vspace{3mm}\\
\noindent \textbf{Acknowledgements}. The first author is supported by the FAPESP postdoctoral grant number 2012/07481-1. The second author is supported by the FAPESP postdoctoral grant number 2011/21398-7. Both authors would like to thank Prof. Rosa Maria Mir\'o-Roig for introducing them to the topic and Prof. Marcos Jardim for many helpful conversations and suggestions.

\section{Preliminaries}

In this section, we fix the notation that will be used in this work and we recall some basics definitions and results. \\
Let $\KK$ be an algebraically closed field of characteristic 0 and let $R = \KK[x_0, \cdots, x_n]$ be the ring of polynomials in $n+1$ variables. Let $M$ be a graded $R-$module.

An $R-$module $N \neq 0$ is said to be a {\it $k-$syzygy} of $M$ if there is an exact sequence of graded $R-$modules
$$
\xymatrix{0 \ar[r] & N \ar[r] & F_k \ar[r]^{\varphi_k} \ar[r] & F_{k-1} \ar[r] & \cdots \ar[r] & F_1 \ar[r]^{\varphi_1}\ar[r] & M \ar[r]  & 0 }$$
where the modules $F_i$ are free $R-$modules.\\
We say that $M$ has a {\it finite projective dimension} if there exist a free resolution over $R$
\begin{eqnarray}\label{minimal}
\xymatrix{0 \ar[r] & F_s \ar[r]^{\varphi_S} & F_{s-1} \ar[r] & \cdots \ar[r] & F_1 \ar[r]^{\varphi_1} & F_0 \ar[r] & M \ar[r] & 0  }
\end{eqnarray}
The least lenght  $s$ of such resolutions is called the {\it projective dimension} of $M$ and denoted by ${\rm pd}(M)$. The resolution (\ref{minimal}) is {\it minimal} if ${\rm im}\,\varphi_i \subset m F_{i-1}$, $ \forall\, i$, where $m = (x_0, \cdots, x_n)$ is the irrelevant ideal of $R$. From the Hilbert syzygy Theorem, see for example \cite[Theorem 1.1.8]{miro}, we have that ${\rm pd}(M) \leq n+1$.\\
If $M$ has a graded minimal free resolution
$$
\xymatrix{0 \ar[r] & \oplus_{j \in \mathbb{Z}}R^{\beta_{p,j}(M)}(-j) \ar[r] & \cdots \ar[r] & \oplus_{j \in \mathbb{Z}} R^{\beta_{1,j}(M)}(-j) \ar[r] & \oplus_{j \in \mathbb{Z}} R^{\beta_{0,j}(M)}(-j) \ar[r] & M \ar[r] & 0 }$$
then the integers  $\beta_{i,j}(M) = \dim {\rm Tor}^R_i (M, \KK)_j$ are called the $(i,j)-$th graded Betti number of $M$, and $\beta_i := \sum_j \beta_{i,j}(M)$ is the $i-$th total Betti number of $M$.\\
We say $M$ has a {\it pure resolution of type} $d = (d_0, \cdots, d_p) $ if it is given by
$$
\xymatrix{0 \ar[r] & R^{\beta_p}(-d_p) \ar[r] & \cdots \ar[r] & R^{\beta_1}(-d_1) \ar[r] & R^{\beta_0}(-d_0) \ar[r] & M \ar[r] & 0}
$$
with $d_0 < d_1 < \cdots<d_p, d_i \in \mathbb{Z}$.

We say $M$ has a {\it linear resolution} if it has a pure resolution of type $(0, 1, \cdots, p)$.

Eisenbud and Schreyer \cite{ES} proved the following result conjectured by Boij and Soderberg \cite{boijS}.

\begin{Theorem}\label{bsconj} For any degree sequence $d = (d_0,\cdots, d_p)$ there is a Cohen-Macaulay module $M$ with a pure resolution of type $d$.

\end{Theorem}

Consider $S = \KK[x_0, \cdots,x_m]$ the ring of polynomials in $m+1$ variables. Let $d = (d_0, \cdots, d_p)$ be a degree sequence. Then by Theorem \ref{bsconj}, there is a Cohen-Macaulay module $M$ with pure resolution

$$\xymatrix{0 \ar[r] & S^{\beta_p}(-d_p) \ar[r] & \cdots \ar[r] & S^{\beta_1}(-d_1) \ar[r] & S^{\beta_0}(-d_0) \ar[r] & M \ar[r] & 0}.
$$ Let $\overline{M}$ be the Artinian reduction of $M$. Then $\overline{M}$ is an Artinian module with pure resolution

\begin{eqnarray}\label{respura}
\xymatrix{0 \ar[r] & R^{\beta_p}(-d_p) \ar[r] & \cdots \ar[r] & R^{\beta_1}(-d_1) \ar[r] & R^{\beta_0}(-d_0) \ar[r] & \overline{M} \ar[r] & 0}
\end{eqnarray} where $R= \KK[x_0,\cdots, x_n]$ with $n = p-1$.\\
We now want to pass from modules to vector bundles and study pure resolutions involving them. Assume we have an Artinian module $\overline{M}$ with pure resolution (\ref{respura}); sheafifying the complex we obtain

\begin{eqnarray} \label{primares}
\xymatrix{0 \ar[r]& \OO_{\PP^n}^{\beta_{n+1}}(-d_{n+1}) \ar[r] & \cdots \ar[r] & \OO_{\PP^n}^{\beta_1}(-d_1) \ar[r] & \OO^{\beta_0}_{\PP^n}(-d_0) \ar[r] &  0}.
\end{eqnarray}
During this paper, we will be interested in such resolution, in particular we will be interested in studying properties of the bundles coming from by splitting the resolution in short exact sequences.
\begin{Definition}
We will call \emph{syzygy bundles} the vector bundles which arise by splitting of resolutions of the type (\ref{primares}).
\end{Definition}

We conclude this section recalling the following notions and results on vector bundles.\\
Let $E$ be a vector bundle on $\mathbb{P}^{n}$. A {\it resolution} of $E$ is an exact sequence

\begin{eqnarray}\label{defresol}
\xymatrix{0 \ar[r]& F_d \ar[r] & F_{d-1} \ar[r] & \cdots  \ar[r] & F_1 \ar[r] & F_0 \ar[r] & E \ar[r]& 0}
\end{eqnarray}  where every $F_i$ splits as a direct sum of line bundles.

One can show that every vector bundle on $\pn$ admits resolution of the form (\ref{defresol}), see \cite[Proposition 5.3]{JM}.  The minimal number $d$ of such resolution is called {\it homological dimension} of $E$, and it is denoted by ${\rm hd}(E)$.  Bohnhorst and Spindler proved the following two results, \cite[Proposition 1.4]{BS} and \cite[Corollary 1.7]{BS}, respectively.

\begin{Proposition}{\label{bsprop}}
Let $E$ be a vector bundle on $\pn$. Then
$$ {\rm hd}(E) \leq d \Longleftrightarrow H^{q}_{*}(E) = 0, \forall \; 1 \leq q \leq n-d-1.$$
\end{Proposition}

\begin{Proposition}\label{bsrk}
Let $E$ be a non splitting vector bundle on $\pn$. Then
$$ {\rm rk}(E) \geq n+1 - {\rm hd}(E).$$
\end{Proposition}

We now recall the notion of \emph{cokernel bundles} and \emph{Steiner bundles}, as defined respectively in \cite{bramb} and \cite{miro-soar}.
\begin{Definition}\label{def-cok}
Let $E_0$ and $E_1$ be two vector bundles on $\PP^n$, with $n\geq2$. A cokernel bundle of type $(E_0,E_1)$ on $\PP^n$ is a vector bundle $C$ with resolution of the form
$$
0 \longrightarrow E_0^a \longrightarrow E_1^b \longrightarrow C \longrightarrow 0
$$
where $b \rk E_1 - a \rk E_0 \geq n$, with $a,b \in \mathbb{N}$, and $E_0$, $E_1$ satisfy the following conditions:
\begin{itemize}
\item $E_0$ and $E_1$ are simple;
\item $\Hom(E_1,E_0) = 0$;
\item $\Ext^1(E_1,E_0) = 0$;
\item the bundle $E_0^\lor \otimes E_1$ is globally generated;
\item $W = \Hom(E_0,E_1)$ has dimension $w \geq 3$.
\end{itemize}
If, moreover,
$$
Ext^i(E_1,E_0) = 0, \:\:\mbox{for each}\:\: i \geq 2
$$
and
$$
Ext^i(E_0,E_1) = 0, \:\:\mbox{for each}\:\: i \geq 1,
$$
the pair $(E_0,E_1)$ is called strongly exceptional and the bundle $C$ is called a Steiner bundle of type $(E_0,E_1)$ on $\PP^n$.

\end{Definition}
\section{Simplicity of syzygy bundles}

In this section we will give some results that ensure the simplicity of the syzygy bundles of the following pure resolution
\begin{equation}\label{seq-orig}
\xymatrix{0 \ar[r] & \OO_{\PP^n}^{\beta_{n+1}}(-d_{n+1}) \ar[r] & \OO_{\PP^n}^{\beta_{n}}(-d_{n}) \ar[r]  \cdots\ar[r] & \OO_{\PP^n}^{\beta_{1}}(-d_{1}) \ar[r] & \OO_{\PP^n}^{\beta_0} \ar[r] & 0}
\end{equation} with $d_1<d_2<\ldots<d_{n+1}$ and $d_i>0$ for each $i$, which splits in short exact sequences

\begin{equation}\label{short-orig}
\begin{array}{c}
0 \longrightarrow \OO_{\PP^n}^{\beta_{n+1}}(-d_{n+1}) \longrightarrow \OO_{\PP^n}^{\beta_{n}}(-d_{n}) \longrightarrow G_1 \longrightarrow 0\\
\vdots\\
0 \longrightarrow G_i \longrightarrow \OO_{\PP^n}^{\beta_{n-i}}(-d_{n-i}) \longrightarrow G_{i+1} \longrightarrow 0\\
\vdots\\
0 \longrightarrow G_{n-1} \longrightarrow \OO_{\PP^n}^{\beta_{1}}(-d_{1}) \longrightarrow \OO_{\PP_n}^{\beta_0} \longrightarrow 0
\end{array}
\end{equation}
We will also consider the dual resolution of (\ref{seq-orig}) and tensor it by $\OO_{\PP^n}(-d_{n+1})$, obtaining
\begin{equation}\label{seq-dual}
\xymatrix{
0 \ar[r] & \OO_{\PP^n}(-d_{n+1})^{\beta_0} \ar[r] & \OO_{\PP^n}^{\beta_{1}}(d_1-d_{n+1}) \ar[r] & \cdots \ar[r]& \OO_{\PP^n}^{\beta_{n}}(d_{n}-d_{n+1}) \ar[r] & \OO_{\PP^n}^{\beta_{n+1}} \ar[r] & 0
}
\end{equation}
which splits as
\begin{equation}\label{short-dual}
\begin{array}{c}
0 \longrightarrow \OO_{\PP^n}^{\beta_0}(-d_{n+1}) \longrightarrow \OO_{\PP^n}^{\beta_{1}}(d_1-d_{n+1}) \longrightarrow F_1 \longrightarrow 0\\
\vdots\\
0 \longrightarrow F_j \longrightarrow \OO_{\PP^n}^{\beta_{j+1}}(d_{j+1}-d_{n+1}) \longrightarrow F_{j+1} \longrightarrow 0\\
\vdots\\
0 \longrightarrow F_{n-1} \longrightarrow \OO_{\PP^n}^{\beta_{n}}(d_{n}-d_{n+1}) \longrightarrow \OO_{\PP_n}^{\beta_{n+1}} \longrightarrow 0
\end{array}
\end{equation}
Let us notice that we have supposed, without loss of generality, that $d_0=0$; else we can tensor the resolution (\ref{primares}) by $\OO_{\PP^n}(-d_0)$ in order to obtain a new resolution as in (\ref{seq-orig}).
Let us prove now some results which ensure the simplicity of the bundles $F_i$, for $i$ from 1 to $n-1$, in particular, the next theorem tell us when the syzigies are simple only looking at the first or the last Betti number.

\begin{Theorem}\label{thm-b0}
Consider a pure resolution as in (\ref{seq-orig}). If $\beta_0 = 1$ or $\beta_{n+1}=1$, then all bundles $F_i$, for $i$ from 1 to $n-1$, are simple.
\end{Theorem}
\begin{proof}
Let us consider first the case $\beta_0=1$, whose importance will be explained by Corollary \ref{cor-ACM}.\\
Let us prove first that the bundle $F_1$ is simple.\\
Consider the exact sequence, obtained by (\ref{short-dual}),
$$
0 \longrightarrow (F^\lor_1)^{\beta_0}(-d_{n+1}) \longrightarrow (F^\lor_1)^{\beta_{1}}(d_1-d_{n+1}) \longrightarrow F^\lor_1 \otimes F_1 \longrightarrow 0
$$
which induces the long exact sequence in cohomology
$$
0 \longrightarrow H^0((F_1^\lor)^{\beta_0}(-d_{n+1})) \longrightarrow H^0((F_1^\lor)^{\beta_1}(d_1-d_{n+1})) \longrightarrow H^0(F_1^\lor \otimes F_1) \longrightarrow H^1((F^\lor_1)^{\beta_0}(-d_{n+1})) \longrightarrow \cdots
$$
Taking again the short exact sequences
$$
0 \longrightarrow F^\lor_{i+1}(d_1-d_{n+1}) \longrightarrow \OO_{\PP^n}^{\beta_{i+1}}(d_1-d_{i+1}) \longrightarrow F^\lor_{i}(d_1-d_{n+1}) \longrightarrow 0,
$$
for $i$ from 1 to $n-2$, we obtain the following chain of isomorphisms
$$
H^0(F_1^\lor(d_1-d_{n+1})) \simeq H^1(F_2^\lor(d_1-d_{n+1})) \simeq \cdots \simeq H^{n-2}(F_{n-1}^\lor(d_1-d_{n+1}))= 0,
$$
Combining these two results, we get an injective map $$H^0(F_1^\lor \otimes F_1) \hookrightarrow H^1((F_1^\lor)^{\beta_0}(-d_{n+1})).$$\\
Consider the short exact sequence
\begin{equation}\label{firstcohom}
0 \longrightarrow F_1^\lor(-d_{n+1}) \longrightarrow \OO^{\beta_1}_{\PP^n}(-d_1) \longrightarrow \OO_{\PP^n}^{\beta_0} \longrightarrow 0
\end{equation}
from which it is straightforward to obtain $\beta_0 = h^1\left(F_1^\lor(-d_{n-1})\right)$; this implies, in the case $\beta_0 = 1$, that the $F_1$ is a simple bundle.\\
Let us prove now that each bundle $F_i$ is simple, for $i$ from $2$ to $n-1$.\\
Consider the following exact sequence
\begin{equation}\label{step-i}
0 \longrightarrow F_{i-1} \otimes F_i^\lor \longrightarrow (F_i^\lor)^{\beta_i}(d_i-d_{n+1}) \longrightarrow F_i \otimes F_i^\lor \longrightarrow 0
\end{equation}
Taking the exact sequences of type
$$
0 \longrightarrow F^\lor_{i+1}(d_i-d_{n+1}) \longrightarrow \OO_{\PP^n}^{\beta_{i+1}}(d_i-d_{i+1}) \longrightarrow F^\lor_{i}(d_i-d_{n+1}) \longrightarrow 0
$$
for $i$ from 1 to $n-2$, and their induced long exact sequence in cohomology, we get a chain of isomorphisms of type
\begin{equation}\label{isom}
H^0(F_i^\lor(d_i-d_{n+1})) \simeq H^1(F_{i+1}^\lor(d_i-d_{n+1})) \simeq \cdots \simeq H^{n-1-i}(F_{n-1}^\lor(d_i-d_{n+1}) = 0.
\end{equation}
Therefore, inducing the long exact sequence in cohomology of (\ref{step-i}), we have an inclusion of type $$H^0(F_i\otimes F_i^\lor) \hookrightarrow H^1(F_{i-1}\otimes F_i^\lor).$$
Proceeding step by step, lowering by one the value of $i$, and using similar isomorphisms as in (\ref{isom}),  that are consequence of the short exact sequences in (\ref{short-dual}), we manage to obtain the following inclusions
$$
H^1(F_{i-1}\otimes F_i^\lor) \hookrightarrow H^2(F_{i-2}\otimes F_i^\lor)\hookrightarrow \cdots \hookrightarrow H^{i-1}(F_{1}\otimes F_i^\lor) \hookrightarrow H^i(F_i^\lor(-d_{n+1})).
$$
In order to compute the last cohomology group, we consider, as before, the exact sequences of the following form
$$
0 \longrightarrow F^\lor_{i+1}(-d_{n+1}) \longrightarrow \OO_{\PP^n}^{\beta_{i+1}}(-d_{i+1}) \longrightarrow F^\lor_{i}(-d_{n+1}) \longrightarrow 0
$$
for $i$ from 1 to $n-2$, obtaining
$$
H^i(F_i^\lor(-d_{n+1})) \simeq H^{i-1}(F_{i-1}^\lor(-d_{n+1})) \simeq \cdots \simeq H^1(F_1^\lor(-d_{n+1})) \simeq \mathbb{C}.
$$
This proves that the bundle $F_i$ is simple.\\
The case $\beta_{n+1}=1$ can be proved, by duality, applying the same technique. Indeed, we can define $\tilde{d}_i = d_{n+1} - d_{n+1-i}$ and, dualizing the resolution (\ref{seq-dual}) and tensoring by $\OO_{\PP_n}(-\tilde{d}_{n+1})$ we obtain a new resolution of the form
\begin{equation}
\xymatrix{
0 \ar[r] & \OO_{\PP^n}(-\tilde{d}_{n+1}) \ar[r] & \OO_{\PP^n}^{\beta_{n}}(\tilde{d}_1-\tilde{d}_{n+1}) \ar[r] &  \cdots \ar[r] & \OO_{\PP^n}^{\beta_{1}}(\tilde{d}_{n}-\tilde{d}_{n+1}) \ar[r] & \OO_{\PP^n}^{\beta_0} \ar[r] & 0
}
\end{equation}
where, as before, the integers $\tilde{d}_i$ satisfy $\tilde{d}_{n+1} > \tilde{d}_n > \ldots > \tilde{d}_1 >0$ and we apply the previous technique.
\end{proof}

As a corollary, we have the following.

\begin{Corollary}\label{cor-ACM}
Consider a quotient ring $A = R/I$ where $I$ is Artinian module and its pure resolution. Then each vector bundle $F_i$, arising from the splitting of the resolution in short exact sequences, is simple.
\end{Corollary}
\begin{proof}
Since $A$ is a quotient, we have that $\beta_0 =1$, then we can apply the previous theorem and obtain that all the bundles $F_i$ are simple.
\end{proof}

With the next lemmas, we give an explicit description and boundaries for the Betti number arising in the resolution we are considering, which will be useful to prove a different theorem about simplicity of the syzygies.

\begin{Lemma} \label{lem-coeff}
The syzygies in the short exact sequences (\ref{short-dual}) satisfies
$$
h^{0}(F_{i-1}^\lor(d_i-d_{n+1})) = \beta_i, \; for \; i =2, \cdots,n.
$$
\end{Lemma}\label{tits-h}
\begin{proof} Twisting the short exact sequences we have
$$
\xymatrix{0 \ar[r] & F^{\vee}_i (d_i-d_{n+1}) \ar[r] & \OO_{\mathbb{P}^n}^{\beta_i} \ar[r] & F^{\vee}_{i-1}{(d_i-d_{n+1})} \ar[r] & 0 }.
$$
One can check that
$$ h^{0}(F_i^\lor(d_i -d_{n+1})) = h^1(F_i^\lor(d_i-d_{n+1}))$$ and the result follows.
\end{proof}

\begin{Lemma}\label{hd} Consider the syzygies $G_i$ and $F_i$ from the short exact sequences (\ref{short-orig}) and (\ref{short-dual}). Then hd$(G_i) = $hd$(F_i) = i$, for $i=1, \cdots, n-1$.

\end{Lemma}
\begin{proof} Let us prove for $F_i$. The case of $G_i$ is analogous. We prove it by induction on $i$. From the short exact sequences (\ref{short-dual}), it is clear that hd$(F_1) = 1$. Let us suppose that hd$(F_{i-1}) = i-1$. We know that hd$(F_i) \leq i$. Suppose hd$(F_i) \leq i-1$. By Proposition \ref{bsprop},
$$H^{q}_{*}(F_i) = 0, \forall\; 1 \leq q \leq n-i.$$ Since $H^{n-i}_{*}(F_i) \simeq H^{n-i+1}_{*}(F_{i-1})$ from the sequences (\ref{short-dual}), and hd$(F_{i-1}) = i-1$, there exists $t \in \mathbb{Z}$ such that $H^{n-i+1}(F_{i-1}(t)) \neq 0$. Therefore hd$(F_i) = i$.

\end{proof}

\begin{Lemma}\label{betti} The Betti numbers $\beta_i$ from the sequence (\ref{seq-orig}) satisfy the inequalities

$$\left.\begin{array}{rcll}
\beta_1 - \beta_0 & \geq & n & \\
\beta_i & \geq & 2n-2i+3, & \mbox{for} \;\;  2  \leq i  \leq \frac{n+1}{2} \\
\beta_i & \geq & 2i+1, & \mbox{for} \;\; \frac{n+1}{2} \leq i \leq n-1\\
\beta_n - \beta_{n+1} & \geq & n &
\end{array}\right.$$ In particular, $\beta_i \geq 3$, for $2 \leq i \leq n-1$.
\end{Lemma}

\begin{proof} We have by Lemma \ref{hd} that hd$(G_i) = $hd$(F_i) = i$, $1 \leq i \leq n-1$. With Proposition \ref{bsrk}

$$ \mbox{rk}(E) \geq n+1 - \mbox{hd}(E) $$ and using the short exact sequences of $F_i$  for $1 \leq i \leq \frac{n+1}{2}$ and the sequences of $G_i$ for $\frac{n+1}{2} \leq i \leq n-1$, we prove the inequalities.

\end{proof}

We are now ready to state the second theorem on the simplicity of the syzygy bundles.

\begin{Theorem}\label{thm-f1} Consider the pure resolution (\ref{seq-dual}) and the syzygies given by the short exact sequences. Then if $F_1$ or $F_{n-1}$ are simple, all the syzygies are simple.
\end{Theorem}

\begin{proof}
Suppose $F_1$ is simple. Let us prove by induction hypothesis that $F_i$ is simple for $i=1, \ldots, n-1$. Consider an injective generic map
$$
\alpha_i :  F_{i-1} \rightarrow \OO_{\PP^n}^{\beta_i}(d_i -d_{n+1}).
$$
We have the following properties:
\begin{itemize}

\item[$(i)$] $F_{i-1}, \OO_{\PP^n}(d_i-d_{n+1})$ are simple, the first bundle by induction hypothesis;

\item[$(ii)$] $\Hom(\OO_{\PP^n}(d_i -d_{n+1}),F_{i-1},) = 0.$ It follows from
$$
H^{0}(F_{i-1}(d_{n+1}-d_i)) \simeq H^j(F_{i-j-1}(d_{n+1}-d_i)), \; 0 \leq j \leq i-2
$$
and $H^{i-2}(F_1(d_{n+1}-d_i)) = H^{i-1}(\OO_{\PP^n}(-d_i)) = 0$;

\item[$(iii)$] $\Ext^1(\OO_{\PP^n}(d_i -d_{n+1}),F_{i-1})=0.$ In fact
$$
H^{1}(F_{i-1}(d_{n+1}-d_i)) \simeq H^j(F_{i-j}(d_{n+1}-d_i)), \; 0 \leq j \leq i-1
$$
and $H^{i-1}(F_1(d_{n+1}-d_i)) = H^{i}(\OO_{\PP^n}(-d_i)) = 0$;

\item[$(iv)$] $F^\lor_{i-1} \otimes \OO_{\PP^n}(d_i -d_{n+1})$ is globally generated. This is clear from the short exact sequences.

\item[$(iv)$] $\dim\Hom(F_{i-1}, \OO_{\PP^n}(d_i-d_{n+1})) \geq 3$. Follows from Lemmas \ref{lem-coeff} and \ref{betti}.
\end{itemize}
Hence we have that $F_{i-1}$ and $\OO_{\PP^n}(d_i-d_{n+1})$ satisfies the conditions of cokernel bundles (see Definition \ref{def-cok}), therefore $\overline{F}_i = {\rm coker} \alpha_i$ is a cokernel bundle and since $q(1,\beta_i) = 1 + {\beta_i}^2 - h^0(F^\lor_{i-1}(d_i-d_{n+1})) \beta_i = 1$ by Lemma \ref{lem-coeff}, we have that $\overline{F_i}$ is simple, see [\cite{bramb},Theorem 4.3].\\
Notice that we have two complexes of the form
$$
\xymatrix{0 \ar[r] & \OO_{\PP^n}^{\beta_0}(-d_{n+1}) \ar[r] & \OO_{\PP^n}^{\beta_1}(d_1-d_{n+1}) \ar[r] & \cdots \ar[r] & \OO_{\PP^n}^{\beta_i}(d_i-d_{n+1}) \ar[r] & F_i \ar[r] & 0  }
$$
and
$$
\xymatrix{0 \ar[r] & \OO_{\PP^n}^{\beta_0}(-d_{n+1}) \ar[r] & \OO_{\PP^n}^{\beta_1}(d_1-d_{n+1}) \ar[r] & \cdots \ar[r] & \OO_{\PP^n}^{\beta_i}(d_i-d_{n+1}) \ar[r] & \overline{F_i} \ar[r] & 0 }$$
therefore $\overline{F_i} \simeq F_i$ and $F_i$ is simple.\\
For the case that $F_{n-1}$ is simple, we can define new coefficients $\tilde{d}_i$, in the same way as in the last part of the proof of Theorem \ref{thm-b0}, and take the dual of (\ref{seq-dual}), tensor it by $\OO_{\PP_n}(-\tilde{d}_{n+1})$ and apply the same technique.

\end{proof}

We would like to find conditions to grant simplicity for every syzygy bundle in the resolution, therefore in the next results, we ask for conditions which give us either $F_1$ or $F_{n-1}$ simple bundles.

\begin{Corollary} Consider the complex (\ref{seq-orig}). If $\beta_n-\beta_{n+1} = n$ or $\beta_1- \beta_0 = n$ then all the syzygies are simple.

\end{Corollary}

\begin{proof} Under this hypothesis, it follows from  \cite[Theorem 2.7]{BS} that either $F_{n-1}$ is stable or $F_1$ is stable, then all the syzygies are simple.

\end{proof}

\begin{Corollary} Consider the complex (\ref{seq-orig}). If the injective map

$$\alpha_{n+1}:  \OO_{\PP^n}^{\beta_{n+1}}(-d_{n+1}) \longrightarrow \OO_{\PP^n}^{\beta_{n}}(-d_{n})$$ is generic and ${\beta_{n+1}}^2 + \beta_{n}^2 -h^0(\OO_{\PP^n}(d_{n+1}-d_{n})) \beta_{n+1} \beta_{n} \leq 1$, or the injective map
 $$\alpha_1^\lor: \OO_{\PP^n}^{\beta_0} \longrightarrow \OO_{\PP^n}^{\beta_1}(d_1)$$ is generic and ${\beta_0}^2 + \beta_{1}^2 - h^0(\OO_{\PP^n}(d_1)) \beta_0 \beta_{1} \leq 1$, then all the syzygies are simple.

\end{Corollary}

\begin{proof} If we have the hypothesis above, ${\rm coker}\alpha_{n+1} = F^\lor_{n-1}(-d_{n+1})$ or ${\rm coker}\alpha^\lor_1 = F_1$ are simple cokernel bundles, see [\cite{bramb},Theorem 4.3],  and the previous theorem applies.

\end{proof}

We conclude this part with the following observation.

\begin{Remark}\label{rmk}
Consider the syzygy modules $N_i$, for $i$ from 1 to $p-2$, which are obtained by the pure resolution
$$
\xymatrix{0 \ar[r] & R^{\beta_p}(-d_p) \ar[r] & \cdots \ar[r] & R^{\beta_1}(-d_1) \ar[r] & R^{\beta_0}(-d_0) \ar[r] & \overline{M} \ar[r] & 0\\
& & N_1 \ar[ur]\\
& 0 \ar[ur]
} $$
Recalling the equivalence of category between modules and their sheafifications, we get that, if the vector bundles $F_i$ are simple, the modules $N_i$ are indecomposable.
\end{Remark}

\section{Exceptionality of syzygy bundles}

In this section we state and prove sufficient and necessary conditions to ensure the exceptionality of the syzygy bundles $F_i$. We obtain the following result.

\begin{Theorem}\label{thm-excep}
Consider the syzygy bundles $F_i$ as defined in (\ref{short-dual}), for $i$ from 1 to $n-1$. Suppose also that $F_i$ are simple for each $i$; then every $F_i$ is exceptional if and only if each one of the following conditions hold
\begin{description}
\item[i)] $\beta_0^2 + \beta_1^2 - \binom{d_1+n}{n}\beta_0\beta_1 = 1$;\\
\item[ii)] $d_1 \leq n;$
\item[iii)]
$$
\left\{
\begin{array}{ll}
H^{n-i+1}(F_{i-1}(d_{n+1}-d_i)) = H^i(F_i^\lor(d_i - d_{n+1})) = 0 & \mbox{if}\:\;n\:\;\mbox{is even;}\vspace{3mm}\\
H^{n-i+1}(F_{i-1}(d_{n+1}-d_i)) = H^i(F_i^\lor(d_i - d_{n+1})) = 0 & \mbox{if}\:\;n\:\;\mbox{is odd and} \;\: i \neq \frac{n+1}{2};\vspace{3mm}\\
H^{n-i+1}(F_{i-1}(d_{n+1}-d_i)) \stackrel{H^i(\varphi)}{\simeq} H^i(F_i^\lor(d_i - d_{n+1})) & \mbox{if}\:\;n\:\;\mbox{is odd and} \;\: i = \frac{n+1}{2}.
\end{array}
\right.
$$
where it will be proven that, if $n$ is odd and $i = \frac{n+1}{2}$, $H^{n-i+1}(F_{i-1}(d_{n+1}-d_i)) \simeq H^i(F_{i-1} \otimes F_i^\lor)$ and the morphism $H^i(\varphi): H^i(F_i^\lor \otimes F_{i-1}) \rightarrow H^i((F_i^\lor)^{\beta_i} (d_i - d_{n+1}))$ is the one obtained by the short exact sequence
$$
0 \longrightarrow F_i^\lor \otimes F_{i-1} \stackrel{\varphi}{\longrightarrow} (F_i^\lor)^{\beta_i} (d_i - d_{n+1}) \longrightarrow F_i^\lor \otimes F_i \longrightarrow 0
$$
considering the long exact induced in cohomology.
\end{description}
\end{Theorem}

\begin{proof}
We will first look for conditions which are equivalent to the excepcionality of the bundle $F_1$; therefore we must compute the cohomology of the bundle $F_1^\lor \otimes F_1$. Consider the short exact sequence
\begin{equation}\label{exc-uno}
0 \longrightarrow (F_1^\lor)^{\beta_0}(-d_{n+1}) \longrightarrow (F_1^\lor)^{\beta_1}(d_1-d_{n+1}) \longrightarrow F_1^\lor \otimes F_1 \longrightarrow 0.
\end{equation}
The first step consists now in calculating the cohomology of $F_1^\lor(-d_{n+1})$. We now consider the sequence
$$
0 \longrightarrow F_1^\lor(-d_{n+1}) \longrightarrow \OO_{\PP^n}^{\beta_1}(-d_1) \longrightarrow \OO_{\PP^n}^{\beta_0} \longrightarrow 0
$$
obtaining that
$$
\begin{array}{ll}
H^k(F_1^\lor(-d_{n+1}) = 0 & \mbox{for}\:\: i=0,2,3,\ldots,n-2,n-1\vspace{3mm}\\
H^1(F_1^\lor(-d_{n+1}) \simeq H^0(\OO_{\PP^n}^{\beta_0}) \simeq \KK^{\beta_0} \vspace{3mm}\\
H^n(F_1^\lor(-d_{n+1}) \simeq H^n(\OO_{\PP^n}^{\beta_1}(-d_1))
\end{array}
$$
We must now compute the cohomology of the second bundle appearing in (\ref{exc-uno}) and we will do it using the short exact sequence
$$
0 \longrightarrow F_1^\lor(d_1-d_{n+1}) \longrightarrow \OO_{\PP^n}^{\beta_1} \longrightarrow \OO_{\PP^n}^{\beta_0}(d_1) \longrightarrow 0
$$
from which we obtain that
$$
H^0(F_1^\lor(d_1-d_{n+1})) \simeq \cdots \simeq H^{n-2}(F_{n-1}^\lor(d_1-d_{n+1})) = 0,
$$
that we have already computed in the proof of Theorem \ref{thm-b0}, and moreover
$$
h^1(F_1^\lor(d_1-d_{n+1})) = dim H^1(F_1^\lor(d_1-d_{n+1})) =\beta_0 \binom{d_1+n}{n} - \beta_1,
$$
$$
H^2(F_1^\lor(d_1-d_{n+1})) \simeq \cdots \simeq H^n(F_1^\lor(d_1-d_{n+1}))  = 0.
$$
From the cohomology we have already calculated, we get that
$$
H^2(F_1^\lor \otimes F_1) \simeq \cdots \simeq H^{n-2}(F_1^\lor\otimes F_1) \simeq H^n(F_1^\lor \otimes F_1) = 0.
$$
Recall that we supposed $F_1$ to be simple, hence we have the following exact sequence in cohomology
$$
0 \longrightarrow \KK \longrightarrow H^1((F_1^\lor)^{\beta_0}(-d_{n+1})) \longrightarrow H^1((F_1^\lor)^{\beta_1}(d_1-d_{n+1})) \longrightarrow H^1(F_1^\lor \otimes F_1) \longrightarrow 0.
$$

Therefore, $H^1(F_1^\lor \otimes F_1)$ vanishes if and only if
$$
\beta_0^2 + \beta_1^2 - \binom{d_1+n}{n}\beta_0\beta_1 = 1.
$$
The other cohomology which we need to vanish is given by
$$
H^{n-1}(F_1^\lor \otimes F_1) \simeq H^n(F_1^\lor(-d_{n+1}) \simeq H^n(\OO_{\PP^n}^{\beta_1}(-d_1)),
$$
that is equal to zero if and only if $d_1\leq n$.\\
Let us suppose the bundle $F_{i-1}$ to be exceptional and let us find conditions ensuring the exceptionality of $F_i$. In order to do so, consider an $i$ fixed from $2$ to $n-1$ and consider also the following short exact sequences
\begin{equation}\label{excep-induction1}
0 \longrightarrow F_i^\lor \otimes F_{i-1} \longrightarrow (F_i^\lor)^{\beta_i}(d_i-d_{n+1}) \longrightarrow F_i^\lor \otimes F_i \longrightarrow 0
\end{equation}
\begin{equation}\label{excep-induction2}
0 \longrightarrow F_i^\lor \otimes F_{i-1} \longrightarrow F_{i-1}^{\beta_i}(d_{n+1}-d_i) \longrightarrow F_{i-1}^\lor \otimes F_{i-1} \longrightarrow 0
\end{equation}
Recall that by induction hypothesis we have that
$$
H^0(F_{i-1}^\lor \otimes F_{i-1}) \simeq \KK \:\:\mbox{and}\:\: H^k(F_{i-1}^\lor \otimes F_{i-1}) = 0 \:\:\mbox{for each}\:\: k >0.
$$
Let us compute the cohomology of the bundle $F_i^\lor(d_i-d_{n+1})$; using the usual short exact sequences defined in (\ref{short-dual}), we have that:
\begin{itemize}
\item if $k<i$ then $H^k(F_i^\lor(d_i-d_{n+1})) \simeq H^{k+n-i-1}(F_{n-1}^\lor(d_i-d_{n+1})) = 0$, because $k+n-1-i \leq n-2$;
\item similarly, if $k>i$ then $H^k(F_i^\lor(d_i-d_{n+1})) \simeq H^{k+1-i}(F_1^\lor(d_i-d_{n+1})) = 0$, because $k+1-i \geq 2$;
\end{itemize}
hence we notice that the only possible non-vanishing cohomology of the bundle is given exactly by $H^i(F_i^\lor(d_i-d_{n+1})).$\\
Using Serre duality and a similar argument, we obtain that
$$
H^k(F_{i-1}(d_{n+1}-d_i)) = 0, \:\:\mbox{if}\:\: k \neq n-i+1
$$
the only possible non-vanishing cohomology of the bundle is given exactly by $H^{n-i+1}(F_{i-1}(d_{n+1}-d_i)).$
Considering this and the induction hypothesis, we get from the exact sequence (\ref{excep-induction2}), that
$$
H^0(F_{i-1}^\lor \otimes F_{i-1}) \simeq H^1(F_i^\lor \otimes F_{i-1}) \simeq \KK
$$
and
$$
H^k(F_i^\lor \otimes F_{i-1}) \simeq H^k(F_{i-1}(d_{n+1}-d_i)) \:\:\forall\:i=2,\ldots,n.
$$
Notice that the following cohomology groups are isomorphic
$$
H^i(F_i^\lor(d_i - d_{n+1})) \simeq H^1(F_1^\lor(d_i-d_{n+1})) \simeq H^{n-1}(F_{n-1}^\lor(d_i-d_{n+1})).
$$
Let us first look for an explicit expression of the group $H^1(F_1^\lor(d_i-d_{n+1}))$ and in order to do so, take the sequence
$$
0 \longrightarrow F_1^\lor(d_i-d_{n+1}) \longrightarrow \OO_{\PP^n}^{\beta_1}(d_i-d_1) \longrightarrow \OO^{\beta_0}_{\PP^n}(d_i) \longrightarrow 0
$$
and we are interested in the following part of the induced sequence in cohomology
$$
0 \longrightarrow H^0(F_1^\lor(d_i-d_{n+1})) \longrightarrow H^0(\OO_{\PP^n}^{\beta_1}(d_i-d_1)) \longrightarrow H^0(\OO^{\beta_0}_{\PP^n}(d_i)) \longrightarrow H^1(F_1^\lor(d_i-d_{n+1})) \longrightarrow 0
$$
The problem now moves to the computation of the dimension of the vector space $H^0(F_1^\lor(d_i-d_{n+1}))$.\\
Consider the exact sequences
$$
\begin{array}{c}
0 \longrightarrow F_2^\lor(d_i - d_{n+1}) \longrightarrow \OO_{\PP^n}^{\beta_2}(d_i - d_2) \longrightarrow F_1^\lor(d_i-d_{n+1}) \longrightarrow 0\\
\vdots\\
0 \longrightarrow F_{i-1}^\lor(d_i - d_{n+1}) \longrightarrow \OO_{\PP^n}^{\beta_{i-1}}(d_i - d_{i-1}) \longrightarrow F_{i-2}^\lor(d_i-d_{n+1}) \longrightarrow 0\vspace{2mm}\\
0 \longrightarrow F_i^\lor(d_i - d_{n+1}) \longrightarrow \OO_{\PP^n}^{\beta_i} \longrightarrow F_{i-1}^\lor(d_i-d_{n+1}) \longrightarrow 0 \vspace{2mm}\\
0 \longrightarrow F_{i+1}^\lor(d_i - d_{n+1}) \longrightarrow \OO_{\PP^n}^{\beta_{i+1}}(\underbrace{d_i - d_{i+1}}_{<0}) \longrightarrow F_i^\lor(d_i-d_{n+1}) \longrightarrow 0
\end{array}
$$
We have already proven that  $H^0(F_i^\lor(d_i-d_{n+1})) = H^1(F_i^\lor(d_i-d_{n+1}))= 0$ and also $H^1(F_j^\lor(d_i-d_{n+1}))=0$ for each $j = 2, \ldots, i+1$, hence we obtain that
$$
h^i(F_i^\lor(d_i-d_{n+1})) = h^1(F_1^\lor(d_i-d_n+1)) = \sum_{k=0}^i (-1)^k \beta_k \binom{d_i-d_k+n}{n}.
$$
Let us now "go to the other side", arriving to $H^{n-1}(F_{n-1}^\lor(d_i-d_{n+1}))$.\\
Consider the short exact sequence
$$
0 \longrightarrow \OO_{\PP^n}^{\beta_{n+1}}(d_i-d_{n+1}) \longrightarrow \OO_{\PP^n}^{\beta_n}(d_i-d_n) \longrightarrow F_{n-1}^\lor(d_i-d_{n+1}) \longrightarrow 0
$$
from which we induce the following part induced in cohomology
$$
0 \longrightarrow H^{n-1}(F_{n-1}^\lor(d_i-d_{n+1})) \longrightarrow H^n(\OO_{\PP^n}^{\beta_{n+1}}(d_i-d_{n+1})) \longrightarrow H^n(\OO_{\PP^n}^{\beta_n}(d_i-d_n)) \longrightarrow H^n(F_{n-1}^\lor(d_i-d_{n+1})) \longrightarrow 0.
$$
As before, take
$$
\begin{array}{c}
0 \longrightarrow F_{n-1}^\lor(d_i-d_{n+1}) \longrightarrow \OO_{\PP^n}^{\beta_{n-1}}(d_i-d_{n+1}) \longrightarrow F_{n-2}^\lor(d_i-d_{n+1}) \longrightarrow 0\\
\vdots\\
0 \longrightarrow F_{i+2}^\lor(d_i-d_{n+1}) \longrightarrow \OO_{\PP^n}^{\beta_{i+2}}(d_i-d_{i+2}) \longrightarrow F_{i+1}^\lor(d_i-d_{n+1}) \longrightarrow 0\vspace{2mm}\\
0 \longrightarrow F_{i+1}^\lor(d_i-d_{n+1}) \longrightarrow \OO_{\PP^n}^{\beta_{i+1}}(d_i-d_{i+1}) \longrightarrow F_{i}^\lor(d_i-d_{n+1}) \longrightarrow 0\vspace{2mm}\\
0 \longrightarrow F_{i}^\lor(d_i-d_{n+1}) \longrightarrow \OO_{\PP^n}^{\beta_i} \longrightarrow F_{i-1}^\lor(d_i-d_{n+1}) \longrightarrow 0
\end{array}
$$
Suppose that $i<n-1$ (or else the computation comes directly considering only the first exact sequence and we will obtain the same result), we have that
$$
H^n(F_i^\lor(d_i-d_{n+1})) \simeq H^{n-1}(F_i^\lor(d_i-d_{n+1})) = 0
$$
and also
$$
H^{n-1}(F_j^\lor(d_i-d_{n+1}))=0, \:\:\mbox{for each}\:\: j=i+1,\ldots,n-2.
$$
We can conclude that
$$
h^i(F_i^\lor(d_i-d_{n+1})) = h^{n-1}(F_{n-1}^\lor(d_i-d_{n+1})) =
\left\{
\begin{array}{ll}
\sum_{k=i+1}^{n+1} (-1)^{k+1} \beta_k \binom{d_k-d_i-1}{n} & \mbox{for} \:\:n\:\:\mbox{even}\vspace{2mm}\\
\sum_{k=i+1}^{n+1} (-1)^{k} \beta_k \binom{d_k-d_i-1}{n} & \mbox{for} \:\:n\:\:\mbox{odd}.
\end{array}
\right.
$$
Let us focus now on the cohomology of the bundle $F_{i-1}(d_{n+1}-d_i)$. \\
We obtain by Serre duality that
$$
H^k(F_{i-1}(d_{n+1}-d_i)) \simeq H^{n-k}(F_{i-1}^\lor(d_i-d_{n+1}-n-1));
$$
therefore, we already know that
\begin{itemize}
\item $H^k(F_{i-1}(d_{n+1}-d_i))=0$ if $k\neq n-i+1$,\\
\item $H^{n-i+1}(F_{i-1}(d_{n+1}-d_i)) \simeq H^{i-1}(F_{i-1}^\lor(d_i-d_{n+1}-n-1)).$
\end{itemize}
As before, we have the isomorphisms
$$
H^{i-1}(F_{i-1}^\lor(d_i-d_{n+1}-n-1)) \simeq H^1(F_1^{\lor}(d_i - d_{n+1} - n- 1)) \simeq H^{n-1}(F_{n-1}^{\lor}(d_i - d_{n+1} - n - 1)).
$$
Using the same techniques as before, if we focus on the first isomorphism, then we have to consider the exact sequences
$$
\begin{array}{c}
0 \longrightarrow F_1^\lor(d_i-d_{n+1}-n-1) \longrightarrow \OO_{\PP^n}^{\beta_1}(d_i-d_1-n-1) \longrightarrow \OO^{\beta_0}_{\PP^n}(d_i-n-1) \longrightarrow 0\vspace{2mm}\\
0 \longrightarrow F_2^\lor(d_i - d_{n+1}-n-1) \longrightarrow \OO_{\PP^n}^{\beta_2}(d_i - d_2-n-1) \longrightarrow F_1^\lor(d_i-d_{n+1}-n-1) \longrightarrow 0\\
\vdots\\
0 \longrightarrow F_{i-1}^\lor(d_i - d_{n+1}-n-1) \longrightarrow \OO_{\PP^n}^{\beta_{i-1}}(d_i - d_{i-1}-n-1) \longrightarrow F_{i-2}^\lor(d_i-d_{n+1}-n-1) \longrightarrow 0\vspace{2mm}\\
0 \longrightarrow F_i^\lor(d_i - d_{n+1}-n-1) \longrightarrow \OO_{\PP^n}^{\beta_i}(-n-1) \longrightarrow F_{i-1}^\lor(d_i-d_{n+1}-n-1) \longrightarrow 0
\end{array}
$$
and knowing that if $i\geq 3$ (or else, as before, I only consider the first short exact sequence and obtain the same result), we have $H^0(F_{i-1}^\lor(d_i-d_{n+1}-n-1)) = H^1(F_{i-1}^\lor(d_i-d_{n+1}-n-1)) = 0$ and therefore
$$
h^{n-i+1}(F_{i-1}(d_{n+1}-d_i)) = h^1(F_1^\lor(d_i-d_{n+1}-n-1)) = \sum_{k=0}^{i-1}(-1)^k \beta_k \binom{d_i-d_k-1}{n}.
$$
Let us focus now on the other isomorphism, computing $H^{n-1}(F_{n-1}^\lor(d_i-d_{n+1}-n-1))$. Take the sequences
$$
\begin{array}{c}
0 \longrightarrow \OO_{\PP^n}^{\beta_{n+1}}(d_i-d_{n+1}-n-1) \longrightarrow \OO_{\PP^n}^{\beta_n}(d_i-d_n-n-1) \longrightarrow F_{n-1}^\lor(d_i-d_{n+1}-n-1) \longrightarrow 0\vspace{2mm}\\
0 \longrightarrow F_{n-1}^\lor(d_i-d_{n+1}-n-1) \longrightarrow \OO_{\PP^n}^{\beta_{n-1}}(d_i-d_{n+1}-n-1) \longrightarrow F_{n-2}^\lor(d_i-d_{n+1}-n-1) \longrightarrow 0\\
\vdots\\
0 \longrightarrow F_{i+2}^\lor(d_i-d_{n+1}-n-1) \longrightarrow \OO_{\PP^n}^{\beta_{i+2}}(d_i-d_{i+2}-n-1) \longrightarrow F_{i+1}^\lor(d_i-d_{n+1}-n-1) \longrightarrow 0\vspace{2mm}\\
0 \longrightarrow F_{i+1}^\lor(d_i-d_{n+1}-n-1) \longrightarrow \OO_{\PP^n}^{\beta_{i+1}}(d_i-d_{i+1}-n-1) \longrightarrow F_{i}^\lor(d_i-d_{n+1}-n-1) \longrightarrow 0\vspace{2mm}\\
0 \longrightarrow F_{i}^\lor(d_i-d_{n+1}-n-1) \longrightarrow \OO_{\PP^n}^{\beta_i}(-n-1) \longrightarrow F_{i-1}^\lor(d_i-d_{n+1}-n-1) \longrightarrow 0
\end{array}
$$
and, being $i-1<n-1$ we can state that $H^{n-1}(F_{i-1}^\lor(d_i-d_{n+1}-n-1)) = H^n(F_{i-1}^\lor(d_i-d_{n+1}-n-1))=0$ and also that
$$
H^{n-1}(F_j^\lor(d_i-d_{n+1}-n-1))=0 \:\:\mbox{for each}\:\:j=i,\ldots,n-2.
$$
We obtain that
$$
h^{n-i+1}(F_{i-1}(d_{n+1}-d_i)) = h^{n-1}(F_{n-1}^\lor(d_i-d_{n+1}-n-1)) =
\left\{
\begin{array}{ll}
\sum_{k=i}^{n+1} (-1)^{k+1} \beta_k \binom{d_k-d_i+n}{n} & \mbox{for}\:\:n\:\:\mbox{even} \vspace{2mm}\\
\sum_{k=i}^{n+1} (-1)^{k} \beta_k \binom{d_k-d_i+n}{n} & \mbox{for}\:\:n\:\:\mbox{odd}
\end{array}
\right.
$$
Let us fix some notation, for each $i$ fixed we will call
$$
\begin{array}{l}
\Sigma_{i,1} = h^i(F_i^\lor(d_i-d_{n+1}))\vspace{2mm}\\
\Sigma_{i,2} = h^{n-i+1}(F_{i-1}(d_{n+1}-d_i)).
\end{array}
$$
We have learned that for each $i$ fixed from $2$ to $n-1$ the cohomology group of $F_i^\lor(d_i-d_{n+1})$ which may not vanish is the $i$-th group, hence the important part of the exact sequence induced in cohomology by (\ref{excep-induction1}) is
\begin{equation}\label{exc-cohom-i}
\begin{array}{c}
\longrightarrow \underbrace{H^{i-1}((F_i^\lor)^{\beta_i}(d_i-d_{n+1}))}_{=0} \longrightarrow H^{i-1}(F_i^\lor \otimes F_i) \longrightarrow H^i(F_{i-1}^{\beta_i}(d_{n+1}-d_i)) \longrightarrow \underbrace{H^i((F_i^\lor)^{\beta_i}(d_i-d_{n+1}))}_{\mbox{dimension}\:\:\beta_i \Sigma_{i,1}} \longrightarrow \\
 \longrightarrow H^{i}(F_i^\lor \otimes F_i) \longrightarrow H^{i+1}(F_{i-1}^{\beta_i}(d_{n+1}-d_i)) \longrightarrow \underbrace{H^{i+1}((F_i^\lor)^{\beta_i}(d_i-d_{n+1}))}_{=0} \longrightarrow
\end{array}
\end{equation}
 We now need to check out how the non-vanishing group in cohomology, associated to the bundle $F_{i-1} (d_{n+1}-d_i)$, relates to the first group, we can have the following situations.\\
 \textbf{Case 1} If $i\neq \frac{n+1}{2}$ and $i\neq\frac{n}{2}$, which means that $n-i \neq i-1$ and $n-i-1 \neq i-1$, then the two groups do not both belong in the sequence (\ref{exc-cohom-i}) and we have
 $$
 H^{n-i+1}(F_{i-1}^{\beta_i}(d_{n+1}-d_i)) \simeq H^{n-i}(F_i^\lor \otimes F_i)
 $$
 and
 $$
 H^i((F_i^\lor)^{\beta_i}(d_i - d_{n+1})) = H^i(F_i^\lor \otimes F_i)
 $$
 hence $F_i$ is exceptional if and only if $H^{n-i+1}(F_{i-1}(d_{n+1}-d_i)) =  H^i(F_i^\lor(d_i - d_{n+1})) = 0$.\\
 \textbf{Case 2} If $i=\frac{n}{2}$, so only in the even cases, we are in the following situation
 $$
0 \longrightarrow \underbrace{H^i((F_i^\lor)^{\beta_i}(d_i-d_{n+1}))}_{\mbox{dimension}\:\:\beta_i \Sigma_{i,1}} \longrightarrow H^{i}(F_i^\lor \otimes F_i) \longrightarrow \underbrace{H^{i+1}(F_{i-1}^{\beta_i}(d_{n+1}-d_i))}_{\beta_i \Sigma_{2,i}} \longrightarrow 0.
 $$
 Being $\Sigma_{p,i}\geq 0$ for $p=1,2$, we can state that $F_i$ is exceptional if and only if $$H^i((F_i^\lor)^{\beta_i}(d_i-d_{n+1}))=H^{i+1}(F_{i-1}^{\beta_i}(d_{n+1}-d_i))=0.$$\\
 \textbf{Case 3} If $i=\frac{n+1}{2}$, so only in the odd cases, we are in the following situation
$$
0 \longrightarrow H^{i-1}(F_i^\lor \otimes F_i) \longrightarrow \underbrace{H^i(F_{i-1}^{\beta_i}(d_{n+1}-d_i))}_{\mbox{dimension}\:\:\beta_i \Sigma_{i,2}} \stackrel{H^i(\varphi)}{\longrightarrow} \underbrace{H^i((F_i^\lor)^{\beta_i}(d_i-d_{n+1}))}_{\mbox{dimension}\:\:\beta_i \Sigma_{i,1}} \longrightarrow H^{i}(F_i^\lor \otimes F_i) \longrightarrow 0,
$$
where $H^i(\varphi)$ is the morphism induced in cohomology by $\varphi: F_i^\lor \otimes F_{i-1} \rightarrow (F_i^\lor)^{\beta_i}(d_i-d_{n+1})$. Therefore $F_i$ is exceptional if and only if $H^{i-1}(F_i^\lor \otimes F_i) = H^i(F_i^\lor \otimes F_i)=0$ if and only if $H^i(\varphi)$ is an isomorphism.\\
This concludes the proof.
\end{proof}

\begin{Corollary}
If each bundle $F_i$, for $i$ from 1 no $n-1$ and defined in (\ref{short-dual}), is a Steiner bundle, i.e. the pair $(F_{i-1}, \OO_{\PP^n}(d_i-d_{n+1})$ is strongly exceptional, where $F_i$ is defined as
$$
0\longrightarrow F_{i-1} \longrightarrow \OO_{\PP^n}^{\beta_i}(d_i-d_{n+1}) \longrightarrow F_i \longrightarrow 0,
$$
and $\beta_0^2 + \beta_1^2 - \binom{d_1+n}{n}\beta_0\beta_1 = 1$; then all bundles $F_i$ are exceptional.
\end{Corollary}

\begin{proof}
The cohomological vanishings appearing in the definition of strongly exceptional pairs of vector bundles, imply the hypothesis (iii) of Theorem \ref{thm-excep}.
\end{proof}

We would like to know if the viceversa of the previous result holds, but at the moment we are only able to state the following.

\begin{Conjecture}
Considering $n$ odd and $i=\frac{n+1}{2}$, if we prove that the two cohomology groups $$H^{n-i+1}(F_{i-1}(d_{n+1}-d_i))\:\:\: \mbox{and} \:\:\:H^i(F_i^\lor(d_i - d_{n+1}))$$ are isomorphic if and only if they are zero; then we would be able to state that the syzygy bundles $F_i$ are Steiner if and only if they are also exceptional.
\end{Conjecture}

As for the results implying simplicity, also for the last theorem we have a correspondent result obtained considering the dual resolution. Recall that the bundles $F_i$ are simple or exceptional if and only if the bundles $G_i$ are.

\begin{Theorem}
Consider the syzygy bundles $G_i$ as defined in (\ref{short-orig}), for $i$ from 1 to $n-1$. Suppose also that $G_i$ are simple for each $i$; then $G_i$, for $i=1,\ldots,n-1$, is exceptional if and only if each one of the following conditions hold
\begin{description}
\item[i)] $\beta_{n+1}^2 + \beta_n^2 - \binom{d_{n+1}-d_n+n}{n}\beta_{n+1}\beta_n = 1$;\\
\item[ii)] $d_{n+1}-d_n \leq n;$
\item[iii)]
$$
\left\{
\begin{array}{ll}
H^{n-i+1}(G_{i-1}(d_{n+1-i})) = H^i(G_i^\lor(-d_{n+1-i})) = 0 & \mbox{if}\:\;n\:\;\mbox{is even;}\vspace{3mm}\\
H^{n-i+1}(G_{i-1}(d_{n+1-i})) = H^i(G_i^\lor(-d_{n+1-i})) = 0 & \mbox{if}\:\;n\:\;\mbox{is odd and} \;\: i \neq \frac{n+1}{2};\vspace{3mm}\\
H^{n-i+1}(G_{i-1}(d_{n+1-i})) \stackrel{H^i(\varphi)}{\simeq} H^i(G_i^\lor(-d_{n+1-i})) & \mbox{if}\:\;n\:\;\mbox{is odd and} \;\: i = \frac{n+1}{2}.
\end{array}
\right.
$$
where, if $n$ is odd and $i = \frac{n+1}{2}$,  we get $H^{n-i+1}(G_{i-1}(d_{n+1-i})) \simeq H^i(G_{i-1} \otimes G_i^\lor)$ and the morphism $H^i(\varphi): H^i(G_i^\lor \otimes G_{i-1}) \rightarrow H^i((G_i^\lor)^{\beta_i} (-d_{n+1-i}))$ is the one obtained by the short exact sequence
$$
0 \longrightarrow G_i^\lor \otimes G_{i-1} \stackrel{\varphi}{\longrightarrow} (G_i^\lor)^{\beta_{n+1-i}} (-d_{n+1-i}) \longrightarrow G_i^\lor \otimes G_i \longrightarrow 0
$$
considering the long exact induced in cohomology.
\end{description}
\end{Theorem}

\section{Examples}

In this section we present some famous pure resolutions and we will apply the results obtained to determine whenever the syzygies are simple or exceptional. Some of these resolutions were studied by \cite{Prata}.

\subsection{Pure linear resolution}

Let $R = \KK[x_0, \ldots,x_n]$ be the ring of polynomials and $I = (x_0, \cdots, x_n)$ be the ideal generated by the coordinate variables . The Koszul complex $K(x_0, \cdots, x_n)$ is given by

$$\xymatrix{0 \ar[r] & R(-n -1)) \ar[r] & R^{\binom{n+1}{n}}(-n) \ar[r] & \cdots \ar[r] &  R^{n+1}(-1) \ar[r] & R \ar[r] & R/I \ar[r] & 0 }$$ Sheafifying we get the exact sequence

\begin{eqnarray}\label{Kcomp}
\xymatrix{0 \ar[r] & \opn(-n-1)) \ar[r] & \opn^{\binom{n+1}{n}}(-n) \ar[r] & \ldots \ar[r]  & \opn^{n+1}(-1) \ar[r] & \opn \ar[r] & 0}.
\end{eqnarray}

\begin{Proposition}
The syzygy bundles arising from the complex (\ref{Kcomp}) are all simple and exceptional.
\end{Proposition}

\begin{proof}
It is a simple computation that the complex satisfies the hypothesis of Theorem \ref{thm-b0} and of Theorem \ref{thm-f1} for simplicity, and the hypothesis of Theorem \ref{thm-excep} for the exceptionality.

\end{proof}

\subsection{Compressed Gorenstein Artinian graded algebras}

Let $I = (f_1, \ldots, f_{\alpha_1}) $ be an ideal generated by $\alpha_1$ forms of degree $t+1$, such that the algebra $A = R/I$ is a compressed Gorenstein Artinian graded algebra of embedding dimension $n+1$ and socle degree $2t$. Thus, by Proposition $3.2$ of \cite{mig-miro-nag}, the minimal free resolution of $A$ is

$$\xymatrix{ 0 \ar[r] & R(-2t-n-1) \ar[r] & R^{\alpha_n}(-t-n) \ar[r] & R^{\alpha_{n-1}}(-t-n+1) \ar[r] & \cdots &}$$
 $$\xymatrix {\cdots \ar[r] & R^{\alpha_p}(-t-p) \ar[r] & \cdots \ar[r] &  R^{\alpha_2}(-t-2) \ar[r] & R^{\alpha_1}(-t-1) \ar[r] & R \ar[r] & A \ar[r] & 0 }$$ where $$\alpha_i = \binom{t+i - 1}{i-1} \binom{t+n+1}{n+1-i} - \binom{t+n-i}{n+1-i} \binom{t+n}{i-1}, \, for  \, i=1, \cdots, n.$$

Sheafifying the complex above we have

\begin{eqnarray}\label{compG}
\xymatrix{ 0 \ar[r] & \OO_{\PP^n}(-2t-n-1) \ar[r] & \OO_{\PP^n}^{\alpha_n}(-t-n) \ar[r] & \OO_{\PP^n}^{\alpha_{n-1}}(-t-n+1) \ar[r] & \cdots}
\end{eqnarray}
\begin{eqnarray*}
\xymatrix{ \cdots \ar[r] & \OO_{\PP^n}^{\alpha_p}(-t-p)  \ar[r] & \cdots \ar[r] &  \OO_{\PP^n}^{\alpha_2}(-t-2) \ar[r] & \OO_{\PP^n}^{\alpha_1}(-t-1) \ar[r]^<<<<{\beta} & \OO_{\PP^n} \ar[r] & 0 }\end{eqnarray*}  where $\beta$ is the map given by the $\alpha_1$ forms of degree $t+1$.

\begin{Remark}  By applying Theorem \ref{thm-b0} , we have that the syzygies $F_i$ of the complex (\ref{compG}) are simple vector bundles. Moreover, ${\rm hd}(F_i) = n-i$, $h^0(F_i^*(-t-i))  =\alpha_i$ for $ 1 \leq i \leq n-1$. If we take $t=1,$ then we get the linear resolution and we already know that all syzygies are exceptional. Nevertheless, it is easy to loose the exceptionality. For instance, if we take $t$ such that $t>n-1$, we do not satisfy the second condition of Theorem \ref{thm-excep}. Moreover, being $\beta_0 = 1$ the first condition of Theorem \ref{thm-excep} is equivalent to prove that $$ \binom{t+n+1}{n} - \binom{t+n-1}{n} = \beta_1 = \binom{d_1+n}{n} = \binom{t+n+1}{n},$$ which are not equal if $t \geq 1$. Hence, for this example, the only exceptional bundles come from the linear resolution.
\end{Remark}

\subsection{Generalized Koszul complex}

The reference for this section is $\cite{miro}$
\begin{Definition}
Let $\mathcal{A}$ be a $p \times q$ matrix with entries in $R$. We say that $\mathcal{A}$ is a $t-$homogeneous matrix if the minors of size $j \times j$ are homogeneous polinomials for all $j \leq t$. The matrix $\mathcal{A}$ is an {\it homogeneous matrix} if their minors of any size are homogeneous.
\end{Definition}
Let $\mathcal{A}$ be an homogeneous matrix. We denote by $I(\mathcal{A})$ the ideal of $R$ generated by the maximal minors of $\mathcal{A}$. Le $\mathcal{A}$ be a $t-$homogeneous matrix. For all $j \leq t$, we denote by $I_j(\mathcal{A})$ the ideal of $R$ generated by the minors of size $j$ of $\mathcal{A}$.\\
Note that to any homogeneous $p \times q$ matrix $\mathcal{A}$, we have a morphism $\varphi : F \rightarrow G $ of free graded $R-$modulos of ranks $p$ and $q$, respectively. We write $I(\varphi) = I(\mathcal{A})$.\\
An homogeneous ideal $I\subset R$ is called {\it determinantal  ideal}  if

\begin{itemize}
\item[$(1)$] there exists a $r-$homogeneous matrix $\mathcal{A}$ of size $p \times q$ with entries in $R$ such that $I = I_t(\mathcal{A})$ and
\item[$(2)$] $ht(I) = (p - r + 1)(q - r + 1).$
\end{itemize}

An homogeneous determinantal ideal $I \subset R$ is called {\it standard determinantal ideal}  if $r = \max\{p,q\}$. That is, an homogeneous ideal  $I \subset R$ of codimension $c$ is called standard determinantal ideal if $I = I_r(\mathcal{A})$ for some homogeneous matrix $\mathcal{A}$ of size $r \times (r+ c-1)$.\\
Let $X \subset \mathbb{P}^{n+c} $, and $\mathcal{A}$ homogeneous matrix associated to $X$. Let $\varphi : F \to G $ be a morphism of free graded $R-$modulos of ranks $t$ and $t+c-1$ respectively, defined by  $\mathcal{A}.$ the generalized Koszul complex $C_i(\varphi^*)$ is given by
$$
\xymatrix{0 \ar[r] & \wedge^{i} G^{*} \otimes S_{0}(F^*) \ar[r] & \wedge^{i-1} G^{*} \otimes S_{1}(F^{*})  \ar[r] & \cdots \ar[r] & \wedge^0 G^{*} \otimes S_{i}(F^{*})\ar[r] & 0}
$$
From this complex we have the complex $D_i(\varphi^*)$
$$
\xymatrix{0 \ar[r] & \wedge^{t+c-1} G^{*} \otimes S_{c-i-1}(F) \otimes \wedge^{t}F \ar[r] &\wedge^{t+c-2} G^{*} \otimes S_{c-i-2}(F) \otimes \wedge^t(F) \ar[r] & \cdots & }
$$
\xymatrix{\cdots \ar[r] &  \wedge^{t+i} G^{*} \otimes S_{0}(F)\otimes \wedge F \ar[r] & \wedge^{i}G^{*} \otimes S_0(F^{*}) \ar[r] & \wedge^{i-1} G^{*} \otimes S_{1}(F^{*}) \ar[r] & \cdots }
$$
\xymatrix{ \cdots \ar[r] & \wedge^{0}G^{*} \otimes S_i F^{*} \ar[r] & 0 }
$$
where $D_0(\varphi^*)$ is called {\it Eagon-Northcott complex} and $D_1(\varphi^*)$ is called {\it Buchsbaum-Rim complex}. \\
Let $\varphi : R(-d)^{a} \rightarrow R^{a+n} $ be a map, let $M$ be the matrix associated to the map and $I = I_a(M)$ be the ideal generated by the maximal minors of $M$. The Eagon-Northcott complex $D_0(\varphi^{*})$ gives us a minimal free resolution of $R/I$

$$\xymatrix{0 \ar[r] & R^{\binom{n+a-1}{n-1}}(-d(n+a)) \ar[r] & R^{(n+a)\binom{n+a-2}{a-1}}(-d(n+a-1)) \ar[r] & \cdots}$$
$$\xymatrix{ \cdots \ar[r] & R^{\binom{a+n}{a}}(-da) \ar[r] & R \ar[r] & R/I \ar[r] & 0}$$

Sheafifying, we get the complex

\begin{eqnarray}\label{EN}
\xymatrix{   0 \ar[r] & \OO_{\PP^n}^{\binom{n+a-1}{a-1}}(-d(n+a)) \ar[r] & \OO_{\PP^n}^{(n+a)\binom{n+a-2}{a-1}}(-d(n+a-1)) \ar[r] & \cdots }
\end{eqnarray}
\begin{eqnarray*}
\xymatrix{\cdots \ar[r] & \OO_{\PP^n}^{\binom{a+n}{a}}(-da) \ar[r] & \OO_{\PP^n} \ar[r] & 0}
\end{eqnarray*}

\begin{Remark} Applying Theorem \ref{thm-b0}, all syzygies $F_i$ of the complex (\ref{EN}) are simple. If we take $d=a=1,$ then we get the linear resolution and we already know that all syzygies are exceptional. We obtain exceptionality, for example, also for $n=3, d=1$ and $a=2$. Nevertheless, it is easy to loose the exceptionality. For instance, if we take $d,a$ such that $da>n$, we do not satisfy the second condition of Theorem \ref{thm-excep}. Moreover, if we consider $n=3$, $d=2$ and $a=1$ the syzygy bundles are not exceptional because the first condition of Theorem \ref{thm-excep} is not satisfied.
\end{Remark}

\nocite{*}
\bibliographystyle{alpha}
\bibliography{pureresbib}





\end{document}